\documentclass{article}

\usepackage{amssymb} 
\usepackage{amsmath} 
\usepackage{latexsym} 
\usepackage{theorem} 
\usepackage{color}
\usepackage{epsfig}
\usepackage{cases}

\theorembodyfont{\itshape} 

\newtheorem{theorem}{Theorem}[section]
\newtheorem{lemma}[theorem]{Lemma}

\newtheorem{corollary}[theorem]{Corollary}
\newtheorem{proposition}[theorem]{Proposition}

\newtheorem{definition}[theorem]{Definition}

\newtheorem{remark}[theorem]{Remark}
\newenvironment{proof}{{\par\addvspace{0.1cm}\noindent \bf Proof. }}{\hfill$\Box$\par\medskip} 
 
\def\n{n}

\def\an{{\lambda-\n}}
\def\a{\alpha}
\def\e{\varepsilon}
\def\la{\lambda}
\def\de{\delta}
\def\Om{\Omega}
\def\pO{\partial\Omega}
\def\AS{\sigma_{n-1}}
\def\RR{\mathbb{R}}

\def\interior#1{\stackrel{\circ}{#1}}

\def\rhosub#1{\mbox{\large $\rho $}_{\mbox{\small $\!{}_{#1}$}}}

\def\R{\Re\mathfrak{e} \,}

\title{\bf Uniqueness of centers of nearly spherical bodies}
\author{Jun O'Hara}

\numberwithin{equation}{section}

\begin{document}

\maketitle

\begin{abstract} 
An $r^{\an}$-center of a compact body $\Om$ in an $n$ dimensional Euclidean space is a point that gives an extremal value of the regularized Riesz potential, which is (Hadamard's regularization of) the integration on $\Om$ of the distance from the point to the power $\an$. 
We show that for any real number $\la$ if a compact body is sufficiently close to a ball in the sense of asphericity then the $r^{\an}$-center is unique. 
We also study the regularized potentials of a unit ball. 
\end{abstract}



\section{Introduction}
A {\em compact body} in $\RR^\n$ is a compact subset of $\RR^\n$ which is a closure of an open subset of $\RR^\n$. 
Let $\Om$ be a compact body with piecewise $C^1$ bounary\footnote{We use Stokes' theorem to derive the boundary integral formulae \eqref{V_alpha_boundary}, \eqref{V0_boundary}. Regularity of the boundary is necessary for Stokes' theorem to be applicable. } $\pO$. 
Let $\interior\Om$ and $\Om^c$ be the interior and the complement of $\Om$ respectively. 
In \cite{O1} the author defined the {\em regularized $r^\an$-potential} $V_\Om^{(\la)}(x)$ of a compact body $\Om$ at a point $x$ in $\RR^n$, where $\la$ is a real number,  by
\begin{numcases}
{\hspace{-0.5cm} V_\Om^{(\la)}(x)=}
\displaystyle \int_\Om{|x-y|}^{\an}\,d\mu(y) & $(0<\la\>\mbox{ or }\>x\in\Om^c)$, \nonumber \\ 
\displaystyle \lim_{\e\to+0}\left(\int_{\Om\setminus B^\n_\e(x)}{|x-y|}^{-\n}\,d\mu(y)-\AS\log\frac1\e\right) & $\big(\la=0, \, x\in\interior\Om\big)$, \label{def_potential_la=0_interior} \\ 
\displaystyle \lim_{\e\to+0}\left(\int_{\Om\setminus B^\n_\e(x)}{|x-y|}^\an\,d\mu(y)-\frac{\AS}{-\la}\cdot\frac1{\e^{-\la}}\right) & $\big(\la<0, \, x\in\interior\Om\big)$,\label{def_potential_la<0_interior}
\end{numcases}
where $\mu$ is the standard Lesbegue measure of $\RR^\n$, $\AS$ is the volume of the $(n-1)$-dimensional unit sphere $S^{\n-1}$ and $B^\n_\e(x)$ is an $\n$-ball with center $x$ and radius $\e$. 
We remark that when $\la\le0$ and $x\in \interior\Om$, $V_\Om^{(\la)}(x)$ is Hadamrd's finite part of a divergent integral $\int_\Om{|x-y|}^{\an}\,d\mu(y)$.  
We do not consider the case when $\la\le0$ and $x\in\pO$. 

In particular, when $\Om$ is convex, $x\in\interior\Om$ and $\la\ne0$,  $V_\Om^{(\la)}(x)$ can be expressed as 
\[
V_\Om^{(\la)}(x)=\frac1\la\int_{S^{\n-1}}\left(\rhosub{\Om{-x}}(v)\right)^\la\,d\sigma(v),
\]
where $\sigma$ is the standard Lebesgue measure of $S^{\n-1}$, $\Om{-x}=\{y-x\,|\,y\in\Om\}$ and $\rhosub{\Om{-x}}:S^{\n-1}\to\RR_{>0}$ is the {\em radial function} 
given by $\mbox{\large $\rho  $}_{\Om{-x}}(v)=\sup\{a\ge0\,|\,x+av\in \Om\}$. 
The regularized potential $V_\Om^{(\la)}(x)$ restricted to the set of convex bodies coincides with the {\em dual mixed volume} %
of $\Om-x$ introduced by Lutwak (\cite{Lu75,Lu88}) up to multiplication by a constant factor. 

In \cite{O1} the author defined an {\em $r^{\an}$-center of $\Om$} by a point where the extremal value of $V_\Om^{(\la)}$ is attained. 
To be precise, it is a point that gives the minimum value of $V_\Om^{(\la)}$ when $\la>\n$, the maximum value of $V_\Om^{(\la)}$ when $0<\la<\n$, and the maximum value of $V_\Om^{(\la)}$ in $\interior\Om$ when $\la\le0$. 
When $\la=\n$, since $V^{(\n)}_\Om(x)$ is constantly equal to ${\rm Vol}(\Om)$, we define {\em $r^0$-center} by a point that gives the maximum value of the log potential 
\[
V^{\rm log}_\Om(x)=\int_\Om\log\frac1{|x-y|}\,d\mu(y). 
\]
For example, the center of mass is an $r^2$-center, and the incenter and circumcenter of a non-obtuse triangle can be considered as ``$r^{-\infty}$-center'' and ``$r^{\infty}$-center'' respectively (\cite{O1}). 
In particular, when $\Om$ is convex, an $r^{\an}$-center $(\la\ne0)$ coincides with the {\em radial center of order $\la$} introduced by Moszy\'nska \cite{M1} (see also \cite{HMP}).
An $r^{\an}$-center exists for any $\la$, but it is not necessarily unique. 
For example, a disjoint union of two balls has at least two $r^{\an}$-centers if $\la$ is sufficiently small. 

It would be natural to look for a sufficient condition for the uniqueness of $r^{\an}$-centers. 
The uniqueness holds if $\la\ge\n+1$ or if $\la\le1$ and $\Om$ is convex (\cite{O1}). 
The argument on the symmetry or the moving plane method (\cite{GNN}) implies that an $n$-ball has the unique $r^{\an}$-center (which coincides with the center in the ordinary sense) for any $\la$. 
In this paper we show that if a compact body $\Om$ is close to an $n$-ball then its $r^{\an}$-center is unique, where we measure the closeness to an $n$-ball by Dvoretzky's  {\em asphericity} (\cite{Dv1, Dv2}) which was originally introduced for convex bodies although we do not assume the convexity of $\Om$. 
To be precise, for any closed interval $[\la_1,\la_2]$ that does not contain $0$ or $\la_1=\la_2=0$ there is a positive number $\e$ such that if a compact body $\Om$ satisfies $B^n_\rho\subset\Om\subset B^n_{(1+\e)\rho}$ for some $\rho>0$ then the $r^\an$-center of $\Om$ is unique for each $\la$ in $[\la_1,\la_2]$. 

As a corollary we obtain the following. Let $\Om_\ell$ be a body which is obtained by inflating $\Om$ by $\ell$, 
\begin{equation}\label{parallel_body}
\Om_\ell=\bigcup_{x\in\Om}B^\n_\ell(x)=\{y\in\RR^n\,:\,\mbox{dist}(y,\Om)\le\ell\}. 
\end{equation}
Then for any closed interval $[\la_1,\la_2]$ satisfying the same condition as above, there is a positive number $C$ such that for any compact body $\Om$ in $\RR^\n$ the $r^{\la-\n}$-center of $\Om_\ell$ is unique if $\ell$ is greater than $C$ times the diameter of $\Om$ and if $\partial \Om_\ell$ is piecewise $C^1$. 
We remark that $\Om_\ell$ is called an {\em $\ell$-parallel body} of $\Om$ when $\Om$ is a convex body. 

In the last section we study the regularized potential of a unit ball. 
We give some properties and give formulae to express the regularized potential in terms of the Gauss hypergeometric functions, which integrates the results of preceding studies that dealt with some of the cases when the potential can be defined without regularization. 
We also give sufficient conditions for the regularized potential of the unit ball to be expressed by elementary functions. 

\medskip 
Acknowledgement. The author would like to thank Shigehiro Sakata for helpful discussions. He would also thank the anonymous reviewer for the careful reading and helpful suggestions.

\section{Proximity to balls}
Put $B^\n_r=B^\n_r(0)$ and $B^\n=B^\n_1$. 
\subsection{Bi-Hausdorff distance, minimal rings, and approximation by balls}
%
The {\em Hausdorff distance} between two non-empty subspaces of $\RR^n$ is defined by \[
d_H(K,L)=\inf\{\e>0\,:\,K\subset L_\e, L\subset K_\e\},
\]
where $K_\e$ and $L_\e$ are given by \eqref{parallel_body}. 
It is frequently used in convex geometry, but it is not suitable for the study of $r^\an$-centers since the centers do not behave continuously with respect to the Hausdorff distance when we deal with non-convex bodies. 
For any small positive number $\e$ one can find a space $X_\e$ with $d_H(B^\n,X_\e)\le\e$ that has at least two $r^\an$-centers if $\la\le0$ whereas $B^\n$ has a unique $r^\an$-center for any $\la\in\RR$. 
This can be done by putting 
\[
X_\e=B^\n\cap\{x=(x_1,\dots,x_n)\in\RR^n\,:\,|x_n|\ge\e\}.
\] 
This examle suggests us that a small crack might cause a big difference for the location of centers. 
To distinguish $B^\n$ and $X_\e$ effectively we shall take into account the Hausdorff distance of the complements. 
\begin{definition} \rm 
Define the {\em bi-Hausdorff distance} between two compact bodies $K$ and $L$ by 
\[
d_{bH}(K,L)=\max\{d_H(K,L), d_H(K^c,L^c)\}. 
\]
\end{definition}
Note that $d_{bH}(B^\n,X_\e)=1$, whereas $d_H(B^\n,X_\e)=\e$. 
\begin{lemma}\label{lemma_bi-Hausdorff}
Let $\Om$ be a compact body. 
We have $d_{bH}(\Om, \rho B^\n)\le\de$ if and only if $(\rho-\de)B^\n\subset\Om\subset(\rho+\de)B^\n$. 
\end{lemma}
\begin{proof}
(1) Suppose $d_{bH}(\Om, \rho B^\n)\le\de$. 
Since $d_H(\Om, \rho B^\n)\le\de$, we have $\Om\subset(\rho B^\n)_\de=(\rho+\de)B^\n$. 
Since $d_H\big(\Om^c, (\rho B^\n)^c\big)\le\de$, we have $\Om^c\subset\big((\rho B^\n)^c\big)_\de=\big((\rho-\de)B^\n\big)^c$, which implies $\Om\supset(\rho-\de)B^\n$. 

(2) Suppose $(\rho-\de)B^\n\subset\Om\subset(\rho+\de)B^\n$. 
Since $\Om\subset(\rho+\de)B^\n$ we have $\Om\subset (\rho B^\n)_\de$, and since $(\rho-\de)B^\n\subset\Om$ we have $B^\n=\big((\rho-\de)B^\n\big)_\de\subset\Om_\de$. 
Therefore $d_H(\Om, \rho B^\n)\le\de$. 
Similarly, as $(\rho-\de)B^\n\subset\Om\subset(\rho+\de)B^\n$ implies 
$\big((\rho-\de)B^\n\big)^c\supset\Om^c\supset\big((\rho+\de)B^\n\big)^c$, 
since $\Om^c\subset \big((\rho-\de)B^\n\big)^c$ we have $\Om^c\subset\big((\rho B^\n)^c\big)_\de$, and since $\big((\rho+\de)B^\n\big)^c\subset \Om^c$ we have $(\rho B)^c
=\big(\big((\rho+\de)B^\n\big)^c\big)_\de\subset\big(\Om^c)_\de$. 
Therefore $d_H\big(\Om^c, (\rho B^\n)^c\big)\le\de$, which completes the proof. 
\end{proof}

Put
\[
r_\Om(x)=\inf_{y\in\Om^c}|y-x|,\>\> R_\Om(x)=\max_{y\in\Om}|y-x| \qquad (x\in\Om). 
\]
Namely, $r_\Om(x)$ is the radius of the largest ball with center $x$ that is contained in $\Om$ if $x$ is an interior point of $\Om$, and $R_\Om(x)$ is the radius of the smallest ball with center $x$ that contains $\Om$. 

Let us apply the notion of minimal rings of Bonnesen and B\'ar\'any to compact bodies which are not necessarily convex. Put $\phi_\Om(x)=R_\Om(x)-r_\Om(x)$. 

\begin{definition} \rm (\cite{Bo,B}) 
A point $x_0$ that gives the minimum value of the function $\phi_\Om=R_\Om-r_\Om\colon\Om\to\RR$ is called the {\em center of minimal ring of $\Om$} and $B^\n_{R_\Om(x_0)}(x_0)\setminus \interior B \!\!{}^{\,\n}_{r_\Om(x_0)}(x_0)$ is called the {\em minimal ring of $\Om$}. 
\end{definition}

Since $\phi_\Om$ is continuous as $\pO$ is piecewise $C^1$, 
the existence of a center of minimal ring holds for any compact body.

When $\Om$ is convex the uniqueness of the center of minimal ring was proved by Bonnesen \cite{Bo} for $n=2$ and by B\'ar\'any \cite{B} in general. 
In this case the ball whose boundary sphere is in the middle of the minimal ring, $B^\n_{(R_\Om(x_0)+r_\Om(x_0))/2}(x_0)$, is a unique ball that gives the best approximation of $\Om$ with respect to the Hausdorff distance (\cite{NS}). 

On the other hand, when $\Om$ is not convex the center of minimal ring is not necessarily unique. 
For example, let $A=(0,0), B=(1,0), C=(1/2,0)$ and $\e=0.1$, and put 
\[
\begin{array}{rcl}
X&=&\displaystyle B^2_1(A)\cup B^2_1(B)\cup
\Big(\big(B^2_2(A)\setminus \interior B\!{}^2_{2-\e}(A)\big)\cap\{(x_1,x_2)\,:\,x_1\ge1/2\}\Big) \\[2mm]
&& \displaystyle \phantom{B^2_1(A)\cup B^2_1(B)} \cup\Big(\big(B^2_2(B)\setminus \interior B\!{}^2_{2-\e}(B)\big)\cap\{(x_1,x_2)\,:\,x_1\le1/2\}\Big)
\end{array}
\]
(Figure \ref{minimalring}). 
\begin{figure}[htbp]
\begin{center}
\includegraphics[width=.25\linewidth]{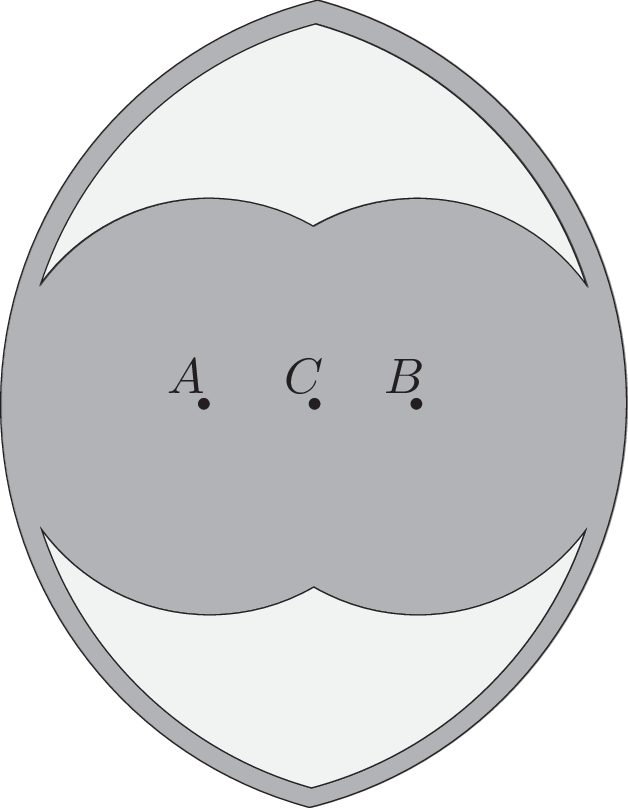}
\caption{The center of symmetry $C$ is not the center of minimal ring}
\label{minimalring}
\end{center}
\end{figure}
Then $\phi_{X}(A)=\phi_X(B)=2-1=1$ whereas $\phi_X(C)=\sqrt{4-(1/4)}-\sqrt{1-(1/4)}>1$. 

Lemma \ref{lemma_bi-Hausdorff} implies that if $x_0$ is a center of minimal ring of a (not necessarily convex) compact body $\Om$ then the ball whose boundary sphere is in the middle of the minimal ring, $B^\n_{(R_\Om(x_0)+r_\Om(x_0))/2}(x_0)$, gives the best approximation of $\Om$ with respect to the bi-Hausdorff distance.

\subsection{Asphericity}

The {\em asphericity} $\a(\Om)$ by Dvoretzky (\cite{Dv1, Dv2}) is defined by 
\[\a(\Om)=\inf_{x\in\interior\Om}\frac{R_\Om(x)}{r_\Om(x)}-1.\]
The infimum above can in fact be replaced by minimum. 
Note that there holds 
\[
\a(\Om)=\inf \big\{\e\,:\, B^\n_r(x)\subset\Om\subset B^\n_{(1+\e)r}(x)\>\mbox{ for some }\>r>0, \, x\in \interior\Om \big\}.
\]
Dudov and Meshcheryakova \cite{Du} showed that if $\Om$ is convex then the function 
\[
\psi\colon\interior\Om\ni x\mapsto \frac{R_\Om(x)}{r_\Om(x)}\in\RR_{\ge1}
\]
is quasi-convex, namely for any real number $b$, $\psi^{-1}((-\infty,b])$ is a convex set. 
A point that gives the minimum value of $\psi$ is called the {\em asphericity center}. 
It is not necessarily unique, as one can see by considering the example of an isosceles triangle of edge lengths $1,a,a$ $(a>1)$, where the set of asphericity centers is the interval on the angle bisector of the smallest angle between two intersection points, one with the bisectors of the longest edges and the other with the angle bisectors of the greatest angles. 
Dudov and Meshcheryakova \cite{Du} showed that the asphericity center is unique if $\Om$ is strictly convex or if $\Om$ is centrally symmetric\footnote{To be precise, Dudov and Meshcheryakova worked with a general norm $n(x)$, and they showed that if $\Om$ is centrally symmetric and the norm $n(x)$ is strictly quasiconvex then the asphericity center is unique. In this article we only use the standard Euclidean norm $\Vert x\Vert$, which is strictly quasiconvex}.

\section{Main theorem}
%
We assume $\n\ge2$ in what follows. 
%
\subsection{Asphericity and the uniqueness of centers}

\begin{theorem}\label{main_thm}
Let $\la_1$ and $\la_2$ $(\la_1\le\la_2)$ be a pair of real numbers such that either $0\not\in[\la_1,\la_2]$ or $\la_1=\la_2=0$. 
Then there is a positive number $\e=\e(\la_1, \la_2)$ such that any compact body $\Om$ in $\RR^\n$ with asphericity smaller than $\e$ has a unique $r^{\an}$-center for each $\la$ in $[\la_1, \la_2]$. 
\end{theorem}

Remark that we can assume $\la_2\le\n+1$ since any compact body has a unique $r^{\an}$-center if $\la\ge \n+1$ as we commented in the introduction. 
Fix $\n$, $\la_1$ and $\la_2$ such that $[\la_1,\la_2]\subset(-\infty,0)$ or $[\la_1,\la_2]\subset(0,n+1]$ or $\la_1=\la_2=0$ in what follows. 

\begin{proof}
The proof consists of two steps. 

First we show there are positive numbers $\rho$ $(0<\rho<1)$ and $r_1$ $(r_1>1)$ such that if a compact body $\Om$ satisfies $B^\n\subset \Om\subset B^\n_{r_1}$ then there is at most one $r^{\an}$-center in $B^\n_\rho$ for each $\la$ in $[\la_1, \la_2]$. 
This follows from Lemma \ref{lemma1} below. 

Next we show there is a positive number $r_2>1$ such that if $B^\n\subset \Om\subset B^\n_{r_2}$ then no point in the complement of $B^\n_\rho$ can be an $r^{\an}$-center for any $\la$ in $[\la_1, \la_2]$. 
This follows from Corollary \ref{coro4} below. 

Then the theorem follows if we put $\e=\min\{r_1,r_2\}$. 
\end{proof}

Let us list some formulae to be used later. 
\begin{itemize}
\item {\rm (\cite{O1} Proposition 2.5)} We have 
\begin{equation}\label{V^negative}
V^{(\la)}_\Om(x)=-\int_{\Om^c}{|x-y|}^{\an}\,d\mu(y) \hspace{0.5cm}\big(\la<0, \, x\in\interior\Om\big).
\end{equation}
\item {\rm (\cite{O1} Proposition 2.3)} Under a homothety $\RR^\n\ni x\mapsto k x\in \RR^\n$ $(k>0)$, 
\begin{equation}\label{homothety}
V_{k\Om}^{(\la)}(kx)=\left\{
\begin{array}{ll}
\displaystyle k^{\la}V_{\Om}^{(\la)}(x) & \hspace{0.3cm}\mbox{if $\la\ne0$},\\[1mm]
\displaystyle V_{\Om}^{(0)}(x)+\AS\log k& \hspace{0.3cm}\mbox{if $\la=0$ and $x\in\interior\Om$,} 
\end{array}
\right.
\end{equation}
where $\AS$ is the volume of the unit $(n-1)$-sphere. 
\item {\rm (\cite{O1} Theorem 2.8)} When $x\not\in\pO$ the potential can be expressed by the boundary integral as 
\begin{numcases}
{V_\Om^{(\la)}(x)=}
\displaystyle \displaystyle \frac1{\la}\int_{\partial\Om} {|x-y|}^{\an}(y-x)\cdot \nu\,d\sigma (y) & $(\la\ne0)$,\label{V_alpha_boundary}\\[1mm]
\displaystyle \displaystyle \int_{\partial\Om} \frac{\log|x-y|}{{|x-y|}^\n}\,(y-x)\cdot \nu\,d\sigma (y) & $(\la=0)$, \label{V0_boundary}
\end{numcases}
where $\nu$ is the unit outer normal vector of $\Om$. 
\end{itemize}

It is useful to consider 
\[
\widehat V_\Om^{(\la)}(x)=\left\{
\begin{array}{ll}
\displaystyle \frac1{\n-\la}\, V_\Om^{(\la)}(x) & \qquad (\la\ne\n), \\[4mm]
\displaystyle V_\Om^{\log}(x)=\int_\Om \log\frac1{|x-y|}\,d\mu(y) & \qquad (\la=\n) 
\end{array}
\right.
\]
to deal with the derivatives of the potentials in a uniform manner. 
Then an $r^\an$-center is a point that gives the maximum value of $\widehat V_\Om^{(\la)}$ on $\RR^n$ $(\la>0)$ or on $\interior\Om$ $(\la\le0)$. 
We replace $r^0/0$ by $\log r$ $(r>0)$ in our formulae 
in what follows, although $\log r=\lim_{a\to0}(r^a-1)/a$ in fact, so that we do not have to deal with the case $\la=\n$ separately. 
Then 
\begin{eqnarray}
\displaystyle \frac{\partial \widehat V_{\Om}^{(\la)}}{\partial x_j}(x)
&=&\displaystyle \int_{\partial\Om}\frac{{|x-y|}^{\an}}{\an}\,e_j\cdot \nu\,d\sigma (y) \hspace{1.4cm}  (x\not\in\pO \>\mbox{ or }\>\la>1)\label{1st_der_boundary}\\
&=&\displaystyle \int_\Om {|x-y|}^{\an-2}(y_j-x_j)\,d\mu(y) \qquad (x\in \Om^c \>\mbox{ or }\>\la>1),\label{1st_der_Omega} \\
\frac{\partial^2 \widehat V_{\Om}^{(\la)}}{\partial x_i\partial x_j}(x)
&=&\displaystyle -\int_{\partial\Om}{|x-y|}^{\an-2}(y_i-x_i)\,e_j\cdot \nu\,d\sigma (y) \qquad (x\not\in\pO \>\mbox{ or }\>\la>2) \label{f_second_partial_derivative_boundary} \\
&=&\displaystyle -\int_{\Om}{|x-y|}^{\an-4}\left((\an-2)(y_i-x_i)(y_j-x_j)+\delta_{ij}{|x-y|}^2\right)d\mu (y) \quad{}\label{f_second_partial_derivative_Omega_bis} \\
&& \hspace{7cm} (x\in\Om^c \>\mbox{ or }\>\la>2), \nonumber
\end{eqnarray}
where $e_j$ is the $j$-th unit vector of $\RR^n$ and $\delta_{ij}$ is Kronecker's delta (cf. \cite{O1} Proposition 2.9, Corollary 2.11). 

\begin{lemma}\label{lemma0}
\begin{enumerate}
\item Assume 
$\Om_1\subset\Om_2$. 
If $x\in\interior\Om_1$ or $x\in\Om_2^c$ or $\la>0$ then $V_{\Om_1}^{(\la)}(x)\le V_{\Om_2}^{(\la)}(x)$. 
\item When $\Om$ is the unit ball, 
\begin{eqnarray}
&&\displaystyle \frac{d}{dt} \widehat V_{B^\n}^{(\la)}(te_j)<0 \qquad \quad (0<t<1 \>\mbox{ or }\> \la>1 \>\mbox{ and }\> 0<t\le1), \label{first_derivative_ball} \\[1mm]
&&\displaystyle \left.\frac{d^2}{dt^2} \widehat V_{B^\n}^{(\la)}(te_j)\right|_{t=0}<0 \label{second_derivative_ball} 
\end{eqnarray}
for any $j$. 
\end{enumerate}
\end{lemma}

We remark that the second derivative of $\widehat V_{B^\n}^{(\la)}(te_j)$ $(0<t<1)$ may be positive when $t$ and $\n$ are sufficiently big. 

\begin{proof}
(1) By Lemma 2.4 of \cite{O1} we have 
\[
V_{\Om_2}^{(\la)}(x)-V_{\Om_1}^{(\la)}(x)=\int_{\Om_2\setminus\Om_1}{|y-x|}^{\la-\n}\,d\mu(y) \ge0
\]
since the integrand is positive. 

(2) \eqref{first_derivative_ball} follows from the comparison of the absolute values of the integrand of \eqref{1st_der_boundary} at a pair of points $y$ and $y'$ which are symmetric in a hyperplane given by $y_j=0$. 

\eqref{second_derivative_ball} follows from the fact that if $i=j$ then the integrand of \eqref{f_second_partial_derivative_boundary} is positive when $y_i\ne0$. 
\end{proof}

\begin{lemma}\label{lemma1}
There are positive numbers $\rho$ $(0<\rho<1)$ and $r_1$ $(r_1>1)$ such that if a compact body $\Om$ satisfies $B^\n\subset \Om\subset B^\n_{r_1}$ then the Hessian of $\widehat V_\Om^{(\la)}$
\[
H\Big(\widehat V_\Om^{(\la)}\Big)(x)=\left(\frac{\partial^2 \widehat V_{\Om}^{(\la)}}{\partial x_i\partial x_j}(x)\right)
\]
is negative definite on $B^\n_\rho$ for any $\la\in[\la_1,\la_2]$. 
\end{lemma}

\begin{proof}
(i) When $\Om$ is the unit ball and $x$ is the origin, by \eqref{f_second_partial_derivative_boundary} we have 
\[
\begin{array}{rcl}
\displaystyle \frac{\partial^2 \widehat V_{B^\n}^{(\la)}}{\partial x_i\partial x_j}(0)&=&\displaystyle 0 \qquad \mbox{ if }i\ne j, \\[4mm]
\displaystyle \frac{\partial^2 \widehat V_{B^\n}^{(\la)}}{\partial x_i^2}(0)&=&\displaystyle -\int_{\partial\Om}{|x-y|}^{\an-2} y_i^{\,2} \,d\sigma (y)<0.
\end{array}
\]
Put 
\[
C_0=\max_{\la\in[\la_1,\la_2], 1\le i\le n} \frac{\partial^2 \widehat V_{B^\n}^{(\la)}}{\partial x_i^2}(0)<0,
\]
then the maximum of the eigenvalues of the Hessian $H\big(\widehat V_{B^\n}^{(\la)}\big)(0)$ is smaller than or equal to $C_0$ for any $\la$ in $[\la_1,\la_2]$. 

\smallskip
Let $F_\Om(\la,x)$ ($\la\in[\la_1,\la_2], \, x\in B_{1/2}^n$) be the maximum of the eigenvalues of the Hessian $H\big(\widehat V_\Om^{(\la)}\big)(x)$. 
Then 
$F_{B^\n}\colon[\la_1,\la_2]\times B_{1/2}^n \to \RR$ is a continuous function. 
Since %
$F_{B^\n}(\la,0)\le C_0<0$ for any $\la\in[\la_1,\la_2]$, and $[\la_1,\la_2]\times \{0\}$ is compact, there are positive numbers $\tilde \e$ and $\rho$ ($\rho\le 1/2$) such that $F_{B^\n}([\la_1,\la_2]\times B^n_{\rho})\subset (-\infty, -2\tilde \e\,]$. 

\smallskip
(ii) Put 
\[
b=\max_{\la\in[\la_1,\la_2],x\in B^n_{\rho},1\le i,j\le \n} 
\left|\frac{\partial^2 \widehat V_{B^\n}^{(\la)}}{\partial x_i\partial x_j}(x)\right|.
\]
Let $\varphi$ be a map from ${[-(b+1),b+1]}^{n(n+1)/2}$ to $\RR$ that assigns to 
$(a_{11},\dots,a_{1n},a_{22},\dots,a_{2n},\dots,a_{nn})$ the maximum of the eigenvalues of a real symmetric matrix $(a_{ij})$. 
Since $\varphi$ is continuous on a compact set, it is uniformly continuous. 
Therefore, for any positive number $\e$ there is a positive number $\de$ $(\de\le1)$ such that if $(a_{ij})$ and $(a_{ij}')$ in ${[-(b+1),b+1]}^{n(n+1)/2}$ satisfy
$|a_{ij}-a_{ij}'|<\de$ for any $i,j$ then $|\varphi(a_{ij})-\varphi(a_{ij}')|<\e$. 

Let $\tilde \de$ be a positive number that corresponds to $\tilde \e$ which was given in (i).

\smallskip
(iii) Since the integrand of \eqref{f_second_partial_derivative_Omega_bis} is bounded when $(\la,x,y)$ belongs to $[\la_1,\la_2]\times B^\n_{\rho}\times \big(B^\n_2\setminus \interior B \!\!{}^\n\big)$, there is $r_1$ ($1<r_1\le 2$) such that 
\begin{equation}\label{estimate_2nd_der_annulus}
\int_{B^\n_{r_1}\setminus B^\n}{|x-y|}^{\an-4}\left|(\an-2)(y_i-x_i)(y_j-x_j)+\delta_{ij}{|x-y|}^2\right|d\mu (y) <\tilde\de
\end{equation}
for any $\la\in[\la_1,\la_2]$, $x\in B^\n_{\rho}$ and for any $i,j$. 

\smallskip
(iv) Suppose $B^\n\subset\Om\subset B^\n_{r_1}$. 
Then \eqref{f_second_partial_derivative_Omega_bis} and \eqref{estimate_2nd_der_annulus} imply that 
\[
\left|
\frac{\partial^2 \widehat V_{\Om}^{(\la)}}{\partial x_i\partial x_j}(x)-
\frac{\partial^2 \widehat V_{B^\n}^{(\la)}}{\partial x_i\partial x_j}(x)
\right|<\tilde\de
\]
for any $\la\in[\la_1,\la_2]$, $x\in B^\n_{\rho}$ and for any $i,j$. 
Then the argument in (ii) implies that 
$
|F_\Om(\la,x)-F_{B^\n}(\la,x)|<\tilde\e
$
for any $\la\in[\la_1,\la_2]$ and $x\in B^\n_{\rho}$. 
Since $F_{B^\n}(\la,x)\le -2\tilde\e$ for any $\la\in[\la_1,\la_2]$ and $x\in B^\n_{\rho}$ by the argument in (i), it follows that $F_\Om(\la,x)<-\tilde\e<0$ for any $\la\in[\la_1,\la_2]$ and $x\in B^\n_{\rho}$, which completes the proof. 
\end{proof}

\begin{lemma}\label{lemma2}
Assume $\la_2<\n$. 
Then there is a positive number $r_3>1$ such that 
\[
V_{B^\n_{r_3}}^{(\la)}(x)<V^{(\la)}_{B^\n}(0) \>\mbox{ if }\>\rho\le |x|< r_3
\]
for any $\la\in[\la_1,\la_2]$, where $\rho$ is a positive number given in Lemma \ref{lemma1}. 
\end{lemma}

\begin{proof} 
The proofs are divided into the following three cases. 

(1) Suppose $\la_1>0$ i.e. $[\la_1,\la_2]\subset(0,n)$. 
First note that 
\[
[\la_1,\la_2]\times[0,1)\ni(\la,t)\mapsto V^{(\la)}_{B^\n}(te_1)\in\RR_{>0}
\]
is continuous. 
Fixing $\la$, Lemma \ref{lemma0} (2) implies that $V^{(\la)}_{B^\n}(te_1)$ is a decreasing function of $t\in[0,1)$. 
Put 
\[
C_1=\min_{\la\in[\la_1,\la_2]}\frac{V^{(\la)}_{B^\n}(0)}{V^{(\la)}_{B^\n}\left(\frac\rho2 \, e_1\right)}, 
\]
then $C_1>1$. 
Take a constant $C_2$ so that $1<C_2<\min\{2^{\la_2},C_1\}$, and put 
$r_3=C_2^{\,1/\la_2}$. 
Then $1<r_3<2$ and $r_3^\la\le r_3^{\la_2}=C_2<C_1$ $(\la\in[\la_1,\la_2])$. 
If $\rho\le t<r_3$, 
then $\rho/2<t/r_3<1$. Using \eqref{homothety} we have
\[
V_{B^\n_{r_3}}^{(\la)}(te_1)
=r_3^\la \, V^{(\la)}_{B^\n}\left(\frac t{r_3}\,e_1\right)
<r_3^\la \, V^{(\la)}_{B^\n}\left(\frac \rho{2}\,e_1\right)
<C_1\, V^{(\la)}_{B^\n}\left(\frac \rho{2}\,e_1\right)
\le V^{(\la)}_{B^\n}(0).
\]

(2) Suppose $\la_2<0$ i.e. $[\la_1,\la_2]\subset (-\infty,0)$. 
The proof is nearly parallel to the above one. Note that 
\[
[\la_1,\la_2]\times[0,1)\ni(\la,t)\mapsto V^{(\la)}_{B^\n}(te_1)\in\RR_{<0}
\]
is continuous. 
Put 
\[
C_1'=\max_{\la\in[\la_1,\la_2]}\frac{V^{(\la)}_{B^\n}(0)}{V^{(\la)}_{B^\n}\left(\frac\rho2 \, e_1\right)}, 
\]
then $0<C_1'<1$. 
Take a constant $C_2'$ so that $1>C_2'>\max\{2^{\la_1},C_1'\}$, and put 
$r_3'={C_2'}^{\,1/\la_1}$. 
Then $1<r_3'<2$ and ${r_3'}^\la\ge {r_3'}^{\la_1}=C_2'>C_1'$ $(\la\in[\la_1,\la_2])$.  %
If $\rho\le t<r_3'$, then $\rho/2<t/r_3'<1$. Using \eqref{homothety} we have 
\[
V_{B^\n_{r_3'}}^{(\la)}(te_1)
={r_3'}^\la \, V^{(\la)}_{B^\n}\left(\frac t{r_3'}\,e_1\right)
<{r_3'}^\la \, V^{(\la)}_{B^\n}\left(\frac \rho{2}\,e_1\right)
<C_1'V^{(\la)}_{B^\n}\left(\frac \rho{2}\,e_1\right)
\le V^{(\la)}_{B^\n}(0).
\]

(3) Suppose $\la_1=\la_2=0$. 
Put $r_3''=\sqrt{1+\rho^2/2}$. 
Then if $r_3''>t\ge\rho$ then \eqref{homothety}, Lemma \ref{lemma0} (2) and \eqref{V^0_up} imply 
\[
V_{B^\n_{r_3''}}^{(0)}(te_1)
=V^{(0)}_{B^\n}\left(\frac t{r_3''}\,e_1\right)+\AS\log r_3''
=\AS \log \sqrt{{r_3''}^{2}-t^2}
\le\AS\log\sqrt{{r_3''}^{2}-\rho^2}
<0=V^{(0)}_{B^\n}(0). 
\]
\end{proof}

\begin{corollary}\label{coro2}
Assume $\la_2<\n$. 
If $B^\n\subset \Om\subset B^\n_{r_3}$ then no point in the complement of $B^\n_\rho$ can be an $r^{\an}$-center for any $\la$ in $[\la_1, \la_2]$, where $\rho$ is a positive number given in Lemma \ref{lemma1}. 
\end{corollary}

\begin{proof}
Assume $|x|\ge\rho$. 
When $\la_2\le0$  
we further assume that $x\in\interior\Om$. 
Then by Lemma \ref{lemma0} (1) we have 
\[
V_\Om^{(\la)}(x)\le\displaystyle V_{B^\n_{r_3}}^{(\la)}(x) 
%
%
<\displaystyle V^{(\la)}_{B^\n}(0) 
\le\displaystyle V_\Om^{(\la)}(0),
\]
which implies that $x$ cannot be an $r^{\an}$-center of $\Om$. 

\end{proof}

\begin{lemma}\label{lemma3}
Assume $\la_1>\n-1$. 
Then there is a positive number $r_4>1$ such that if $B^\n\subset \Om\subset B^\n_{r_4}$ then 
\[
\sup_{\rho\le|x|\le r_4, \la\in[\la_1,\la_2]}\frac{d}{dt}
\left.\widehat V^{(\la)}_{\Om}\left(x+t\frac{x}{|x|}\right)\right|_{t=0}<0,
\]
where $\rho$ is a positive number given in Lemma \ref{lemma1}. 
\end{lemma}

\begin{proof}
By choosing an axes with $e_1=x/|x|$, the inequality above can be reduced to 
\[
\sup_{\rho\le t\le r_4, \la\in[\la_1,\la_2]}\frac{d}{dt}
\widehat V^{(\la)}_{\Om}\left(te_1\right)<0.
\]
\eqref{1st_der_Omega} implies 
\begin{equation}\label{1st_der_t}
\frac{d}{dt}\widehat V^{(\la)}_{\Om}\left(te_1\right)
=\int_\Om {|te_1-y|}^{\an-2}(y_1-t)\,d\mu(y).
\end{equation}
The proof is divided into two cases according to whether $x$ belongs to $B^\n$ or not. Let us first prepare two estimates to be used. 

(i) Since \eqref{1st_der_t} is continuous in $\la$ and $t$, and is negative when $\Om$ is the unit ball and $0<t\le1$ by Lemma \ref{lemma0} (2), if we put 
\[
C_2=\max_{\la\in[\la_1,\la_2], t\in[\rho,1]}
\int_{B^\n} {|te_1-y|}^{\an-2}(y_1-t)\,d\mu(y)
\]
then $C_2<0$. 

There is a positive number $r_5$ $(r_5>1)$ such that 
\[
\int_{A_+(r_5,t)} {|te_1-y|}^{\an-2}(y_1-t)\,d\mu(y)<-C_2
\]
for any $\la_\in[\la_1,\la_2]$ and $t\in[\rho,1]$, where $A_+(r_5,t)$ is the part of the annulus $B^\n_{r_5}\setminus \interior B \!\!{}^\n$ on the right side of point $te_1$:
\[
A_+(r_5,t)=\{y=(y_1,\dots,y_\n)\in\RR^\n\,:\,1\le|y|\le r_5, \,y_1\ge t\}
\]
(Figure \ref{A+}). 
This can be verified by the same argument as in Lemma \ref{lemma1} (iii). 

\smallskip
(ii) Put
\[
\begin{array}{rcl}
C_3&=&\displaystyle \max_{\la_\in[\la_1,\la_2],1\le t\le2}
\int_{B^\n} {|te_1-y|}^{\an-2}(y_1-t)\,d\mu(y)<0, \\[4mm]
C_4&=&\displaystyle \max_{\la_\in[\la_1,\la_2]}\int_{B^\n_+} {|y|}^{\an-2}\,y_1\,d\mu(y)>0,
\end{array}
\]
where $B^\n_+=\{y\in B^\n\,:\,y_1\ge0\}$. 

Put
\[
k=\min\left\{\frac12\min_{\la_\in[\la_1,\la_2]}\left(\frac{-C_3}{C_4}\right)^{1/(\la-1)}, \,
\sqrt3\right\}
\quad \mbox{ and }\quad r_6=\sqrt{1+k^2}.
\]
Remark that $\la-1>\n-2\ge0$ since $\la_1>\n-1$, $r_6\le2$ since $k\le\sqrt3$, and $k^{\la-1}C_4<-C_3$ for any $\la\in[\la_1,\la_2]$. 

Put $B^\n_+(r_6,t)=\{y\in B^\n_{r_6}\,:\,y_1\ge t\}$ for $1\le t\le r_6$. 
If $1\le t\le r_6$ then $B^\n_+(r_6,t)\subset kB^\n_++t$, where $kB^\n_++t=\{y+t\,:\,y\in kB^\n_+\}$ (Figure \ref{fig3}), which implies 
\[
\begin{array}{rcl}
\displaystyle \int_{B^\n_+(r_6,t)} {|te_1-y|}^{\an-2}(y_1-t)\,d\mu(y)
&\le&\displaystyle \int_{kB^\n_++t} {|te_1-y|}^{\an-2}(y_1-t)\,d\mu(y) \\[4mm]
&=&\displaystyle  k^{\la-1}\int_{B^\n_+} {|y|}^{\an-2}\,y_1\,d\mu(y) \\[4mm]
&\le&\displaystyle k^{\la-1}C_4 \\[1mm]
&<&-C_3. 
\end{array}
\]

\smallskip
(iii) Put $r_4=\min\{r_5,r_6\}$. 
Suppose $B^\n\subset\Om\subset B^\n_{r_4}$, $\la\in[\la_1,\la_2]$ and $\rho\le t\le r_4$. 
Put 
\[
H_+(t)=\{y\in\RR^\n\,:\,y_1> t\}, \quad H_-(t)=\{y\in\RR^\n\,:\,y_1< t\}.
\]
Remark that the integrand of \eqref{1st_der_Omega} when $x=te_1$ is positive on $H_+(t)$ and negative on $H_-(t)$. 

\smallskip
(iii-a) Suppose $\rho\le t<1$. 
Since $\Om\cap H_-(t)\supset B^\n\cap H_-(t)$ and $\Om\cap H_+(t)\subset B^\n_{r_4}\cap H_+(t)$ we have 
\[
\begin{array}{rcl}
\displaystyle \frac{d}{dt}\widehat V^{(\la)}_{\Om}\left(te_1\right)
&=&\displaystyle \int_\Om {|te_1-y|}^{\an-2}(y_1-t)\,d\mu(y) \\[4mm]
&\le&\displaystyle \int_{\left(B^\n\cap H_-(t)\right) \cup (B^\n_{r_4}\cap H_+(t))} {|te_1-y|}^{\an-2}(y_1-t)\,d\mu(y) \\[4mm]
&=&\displaystyle \int_{B^\n} {|te_1-y|}^{\an-2}(y_1-t)\,d\mu(y) 
+\int_{A_+(r_4,t)} {|te_1-y|}^{\an-2}(y_1-t)\,d\mu(y) \\[4mm]
&<&0
\end{array}
\]
by the argument in (i) (Figure \ref{A+}). 

\begin{figure}[htbp]
\begin{center}
\begin{minipage}{.45\linewidth}
\begin{center}
\includegraphics[width=\linewidth]{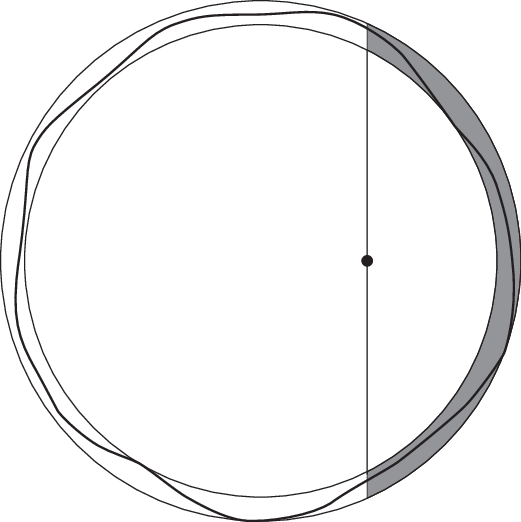}
\caption{$A_+(r_4,t)$}
\label{A+}
\end{center}
\end{minipage}
\hskip 0.4cm
\begin{minipage}{.45\linewidth}
\begin{center}
\includegraphics[width=0.5\linewidth]{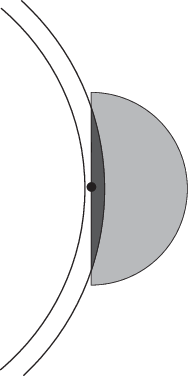}
\caption{$B^\n_+(r_6,t)$ and $kB^\n_++t$}
\label{fig3}
\end{center}
\end{minipage}
\end{center}
\end{figure}

(iii-b) Suppose $1\le t\le r_4$. 
Since $\Om\cap H_-(t)\supset B^\n$ and 
$\Om\cap H_+(t)\subset B^\n_+(r_4,t)\subset B^\n_+(r_6,t)$ 
we have 
\[
\begin{array}{rcl}
\displaystyle \frac{d}{dt}\widehat V^{(\la)}_{\Om}\left(te_1\right)
&\le&\displaystyle \int_{B^\n} {|te_1-y|}^{\an-2}(y_1-t)\,d\mu(y) 
+\int_{B^\n_+(r_6,t)} {|te_1-y|}^{\an-2}(y_1-t)\,d\mu(y)<0 
%
%
\end{array}
\]
by the argument in (ii), which completes the proof. 
\end{proof}

\begin{corollary}\label{coro3}
Assume $\la_1>\n-1$. 
If $B^\n\subset \Om\subset B^\n_{r_4}$ then no point in the complement of $B^\n_\rho$ can be an $r^{\an}$-center for any $\la$ in $[\la_1, \la_2]$, where $\rho$ is a positive number given in Lemma \ref{lemma1}. 
\end{corollary}

\begin{corollary}\label{coro4}
Put $r_2=\min\{r_3, r_4\}$. 
If $B^\n\subset \Om\subset B^\n_{r_2}$ then no point in the complement of $B^\n_\rho$ can be an $r^{\an}$-center for any $\la$ in $[\la_1, \la_2]$, where $\rho$ is a positive number given in Lemma \ref{lemma1}. 
\end{corollary}

\begin{proof}
By Corollaries \ref{coro2} and \ref{coro3} we have only to consider the case when $\n\le\la_2\le\n+1$ and $0<\la_1\le \n-1$. 
Then $[\la_1,\la_2]=[\la_1,\n-1/3]\cup[\n-2/3,\la_2]$; the first interval is treated in Corollary \ref{coro2} and the second by Corollary \ref{coro3}. 
\end{proof}

\subsection{Uniqueness of centers of parallel bodies}
As a corollary of Theorem \ref{main_thm} we obtain
\begin{corollary}
Suppose $\pO$ is of class $C^2$. 
For any pair of real numbers $\la_1, \la_2$ $(\la_1\le\la_2)$ with either $0\not\in[\la_1,\la_2]$ or $\la_1=\la_2=0$ there is a positive number $C=C(\la_1, \la_2)$ such that $\Om_\ell$ defined by \eqref{parallel_body} 
has a unique $r^{\an}$-center for each $\la$ in $[\la_1, \la_2]$ if $\ell$ is greater than $C$ times the diameter of $\Om$ and if $\partial \Om_\ell$ is piecewise $C^1$. 
\end{corollary}

We conjecture that if $\Om$ is piecewise $C^2$ then $\pO_\ell$ is piecewise $C^1$ for any $\ell>0$. 
We remark that when $\Om$ is not convex $\pO_\ell$ may not be $C^1$ even if $\Om$ is of class $C^2$. 
The anonymous reviewer suggested that the $C^2$ regularity condition might be relaxed. 

\begin{proof}
Let $\Om$ be a compact body whose diameter is $d$. 
Let $x$ be a point in the interior of $\Om$. 
Since $\{x\}\subset\Om\subset B^\n_d(x)$, for any positive number $\ell$ we have 
$B^\n_\ell(x)\subset\Om_\ell\subset B^\n_{\ell+d}(x)$, which implies 
$\ell\le r(x,\Om_\ell)\le R(x,\Om_\ell)\le \ell+d$ and hence $\a(\Om_\ell)\le d/\ell$. 
If we put $C(\la_1, \la_2)=1/\e(\la_1, \la_2)$, where $\e(\la_1, \la_2)$ is given in Theorem \ref{main_thm},  Corollary holds. 
\end{proof}

\section{Potential of a unit ball}

In this section we study the regularized potentials of a unit ball and give formulae to express them in terms of the Gauss hypergeometric functions. 
We remark that some of the content of this section has appeared in the literature (see Remark \ref{remark_literature}). 
The consideration of the case where regularization is needed for the definition of the potential is new. 

\subsection{Hypergeometric functions}
Let ${}_2F_1(a,b;c;u)$ be the Gauss hypergeometric function 
\[
{}_2F_1(a,b;c;u)=\sum_{k=0}^{\infty}\frac{(a)_k(b)_k}{k!\,(c)_k}u^k \hspace{0.5cm}(|u|<1),
\]
where $(p)_k$ is the Pochhammer symbol
\[
(p)_k=p(p+1)\dots(p+k-1)=\frac{\Gamma(p+k)}{\Gamma(p)}. 
\]
The Gauss hypergeometric function is the solution of the following hypergeometric differential equation, namely, if we put $w(u)={}_2F_1(a,b;c;u)$ then it satisfies 
\begin{equation}\label{gaussian_hypergeom_eq}
u(1-u)w''+(c-(a+b+1)u)w'-ab\,w=0.
\end{equation}
The derivative is given by 
\begin{equation}\label{hypergeom_derivative}
\frac{d}{du}{}_2F_1(a,b;c;u)=\frac{ab}{c}\,{}_2F_1(a+1,b+1;c+1;u).
\end{equation}
The Gauss summation theorem states 
\begin{equation}\label{Gauss_thm}
{}_2F_1(a,b;c;1)=\frac{\Gamma(c)\Gamma(c-a-b)}{\Gamma(c-a)\Gamma(c-b)}
\end{equation}
if $\R (c-a-b)>0$.

\subsection{Formula of the potential of a unit ball}

Put $V(t)=V^{(\la)}_{B^n}(te_1)$ and $x_t=te_1$ in what follows. 
Then \eqref{V_alpha_boundary}, \eqref{1st_der_boundary} and \eqref{f_second_partial_derivative_boundary} imply  
\begin{eqnarray}
V(t)&=&\displaystyle \frac1\la\int_{S^{n-1}}|te_1-y|^{\la-\n}\,(1-ty_1)\,d\sigma(y) \qquad(\la\ne0), \label{f_V_0}\\[1mm]
V'(t)&=&\displaystyle -\int_{S^{n-1}}|te_1-y|^{\la-\n}\,y_1\,d\sigma(y), \label{f_V_1}\\[1mm]
V''(t)&=&\displaystyle -(\la-\n)\int_{S^{n-1}}|te_1-y|^{\la-n-2}\,(t-y_1)\,y_1\,d\sigma(y)  \notag 
\end{eqnarray}
for $|t|\ne1$. Note that from \eqref{f_V_0} and \eqref{f_V_1} we obtain
\begin{equation}\label{f_int_S_r^{a-n}}
\int_{S^{n-1}}|x_t-y|^{\la-n}\,dy=\la V(t)-t\,V'(t)=\left(\la-t\,\frac{d}{dt}\right)V^{(\la)}_{B^n}(t)
\end{equation}
for $\la\ne0$ and $|t|\ne1$. 

\begin{lemma}\label{lemma_reflection_t}
If $\la\ne0$ and $|t|\ne0,1$ then the following reflection formula holds; 
\begin{equation}\label{f_reflection_potential_t}
V(t)=t^{\la-n}\left[V\left(\frac1t\right)+\frac1\la\left(t-\frac1t\right) V'\left(\frac1t\right)\right].
\end{equation}
\end{lemma}

\begin{proof}
Let $I$ be the inversion in the unit sphere $S^{n-1}$ with center the origin. 
Then $I(y)=y$ $(y\in S^{\n-1})$ and $I(x_t)=x_{1/t}$, where $x_t=te_1$. 

\begin{figure}[htbp]
\begin{center}
\includegraphics[width=.38\linewidth]{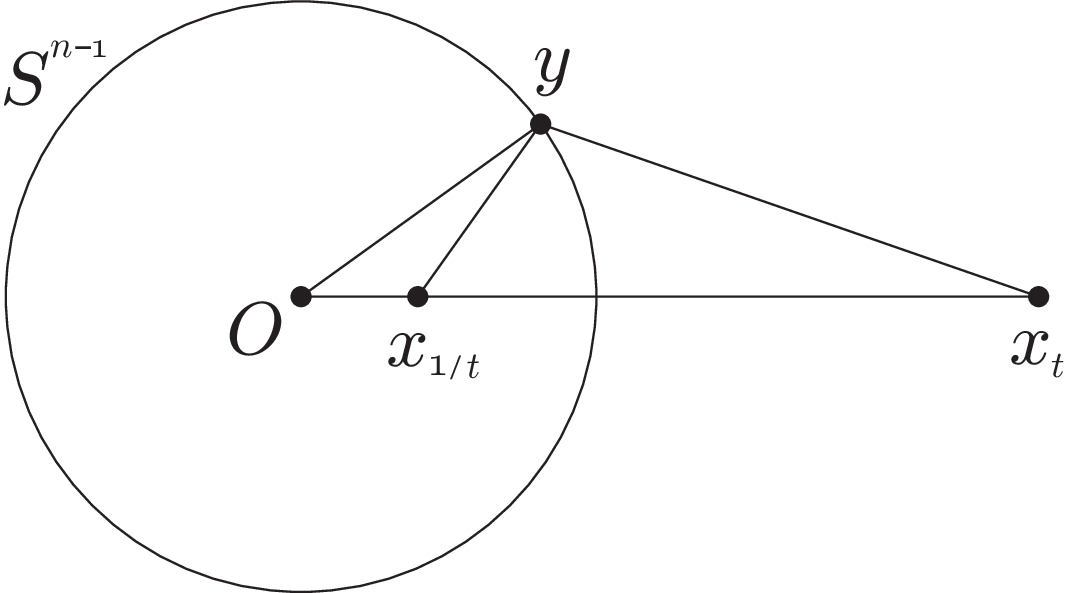}
\caption{$\angle x_tOy=\angle yOx_{1/t}$ and $|Ox_t|:|Oy|=t:1=|Oy|:|Ox_{1/t}|$}
\label{inversion}
\end{center}
\end{figure}

Since $\triangle Ox_ty$ is similar to $\triangle Oyx_{1/t}$ with similarity ration $1/t$ (Figure \ref{inversion}) we have 
\begin{equation}
|x_t-y|
=t\,|x_{1/t}-y|.  \notag
\end{equation}
Therefore, \eqref{f_V_0}, \eqref{f_V_1}, and \eqref{f_int_S_r^{a-n}} yield 
\[\begin{array}{rcl}
V(t)&=&\displaystyle \frac1\la\int_{S^{n-1}}|x_t-y|^{\la-n}(1-ty_1)\,dy\\[3mm]
&=&\displaystyle \frac{t^{\la-n}}\la\int_{S^{n-1}}|x_{1/t}-y|^{\la-n}\,dy
-\frac{t^{\la-n+1}}\la\int_{S^{n-1}}|x_{1/t}-y|^{\la-n}y_1\,dy \\[3mm]
&=&\displaystyle \frac{t^{\la-n}}\la\left[\la V\left(\frac1t\right)-\frac1t\,V'\left(\frac1t\right)\right]+\frac{t^{\la-n+1}}\la\, V'\left(\frac1t\right),
\end{array}
\]
which implies \eqref{f_reflection_potential_t}. 
\end{proof}

\begin{lemma}\label{prop_V_ball_reflection_alpha} 
If $\la\ne0$ and $|t|\ne1$ the regularized Riesz potential of a unit $n$-ball satisfies a reflection formula in $\la$;
\begin{eqnarray}
V^{(\la)}_{B^n}(t)&=&\displaystyle -{\left(1-t^2\right)}^\la V^{(-\la)}_{B^n}(t) \hspace{0.9cm}(|t|<1) \label{f_V_ball_reflection_alpha_t<1} \notag \\[2mm]
V^{(\la)}_{B^n}(t)&=&\displaystyle {\left(t^2-1\right)}^\la V^{(-\la)}_{B^n}(t) \phantom{-}\hspace{0.9cm}(|t|>1) \label{f_V_ball_reflection_alpha_t>1} \notag
\end{eqnarray}
\end{lemma}
\begin{proof} 
(1) Suppose $0\le t<1$. 
Let $I_t$ be an inversion in a unit sphere with center $x_t=te_1$. Then, since 
\[
I_t(e_1)=\left(t+\frac1{1-t}\right)e_1, \> I_t(-e_1)=\left(t-\frac1{1+t}\right)e_1,
\] 
the closure of the complement of $I_t(B^n)$, which we denote by ${(B^n)}_t^\star$, is a ball with center $x_{\hat t}$ and radius $\rho$, where 
\[
\begin{array}{rcl}
\widehat t &=& \displaystyle \frac12\left[\left(t+\frac1{1-t}\right)+\left(t-\frac1{1+t}\right)\right] = t+\frac t{1-t^2}, \\[4mm]
\rho &=& \displaystyle  \frac12\left[\left(t+\frac1{1-t}\right)-\left(t-\frac1{1+t}\right)\right]
= \frac1{1-t^2}.
\end{array}
\]
By Corollary 2.21 of \cite{O1} we have 
\[
V^{(\la)}_{B^n}(x_t)=-V^{(-\la)}_{{(B^n)}_t^\star}(x_t).
\]
Let $g$ be a composition of a homothety with ratio $(1-t^2)^{-1}$ and then a translation by $x_{\hat t}$:
\[
g(x_1,\dots, x_n)=\left(\widehat t+\frac{x_1}{1-t^2},\frac{x_2}{1-t^2},\dots,\frac{x_n}{1-t^2}\right).
\]
Then ${(B^n)}_t^\star=g(B^n)$ and $x_t=g(x_{-t})$. Therefore, \eqref{homothety} implies 
\[
V^{(\la)}_{B^n}(x_t)=-V^{(-\la)}_{{(B^n)}_t^\star}(x_t)
=-\left(\frac1{1-t^2}\right)^{-\la}V^{(-\la)}_{B^n}(x_{-t})
=-{(1-t^2)}^\la V^{(-\la)}_{B^n}(x_{t}).
\]

\medskip
(2) Suppose $t>1$. Since $x_t\not\in B^n$, we do not need regularization in defining the potential $V^{(\la)}_{B^n}(x_t)$. 
Let $I_t, \widehat t$, and $\rho$ be as in (1). Put $\widetilde{B^n}=I_t(B^n)$ and $\widetilde y=I_t(y)$ ($y\in B^n$). Since
\[
|x_t-y|=\frac1{\left|x_t-\widetilde y\,\right|} \quad \mbox{ and } \quad dy=\frac1{{\left|x_t-\widetilde y\,\right|}^{2n}}\,d\widetilde y\,,
\]
we have
\[
V^{(\la)}_{B^n}(x_t) = \int_{B^n} |x_t-y|^{\la-n}\,dy = \int_{\widetilde{B^n}} {\left|x_t-\widetilde y\,\right|}^{-\la-n}\,d\widetilde y
=V^{(-\la)}_{\widetilde{B^n}}(x_t)
\]
On the other hand, since 
$\widetilde{B^n}$ is equal to the image of $B^n$ of a homothety with center $x_t$ and ratio $\rho$, 
we have 
\[
V^{(-\la)}_{\widetilde{B^n}}(x_t)={\rho}^{-\la}\,V^{(-\la)}_{B^n}(x_t)={(t^2-1)}^\la V^{(-\la)}_{B^n}(x_t),
\]
which completes the proof. 
\end{proof}

\begin{lemma}\label{lemma_hypergeom_diff_eq} 
When $\la\ne0$ and $|t|\ne1$ the regularized Riesz potential of the unit ball $V(t)=V^{(\la)}_{B^n}(t)$ satisfies the following differential equation:
\begin{equation}\label{hypgeom_diff__eq_VV'V''}
t(1-t^2)V''(t)+\left(n-1+(2\la-n-1)\,t^2\right)V'(t)-\la(\la-n)\,tV(t)=0. 
\end{equation}
\end{lemma}

\begin{proof}
The equation \eqref{hypgeom_diff__eq_VV'V''} holds for $t=0$ since $V'(0)=0$ by symmetry. 
Therefore we may assume $t>0$. 

Substitution of $t$ in \eqref{f_reflection_potential_t} by $1/t$ yields
\[
V\left(\frac1t\right)=t^{-\la+n}\left[V(t)+\frac1\la\left(\frac1t-t\right)V'(t)\right].
\]
On the other hand, transposition of $V(\frac1t)$ in \eqref{f_reflection_potential_t} yields
\[
V\left(\frac1t\right)=t^{-\la+n}\,V(t)-\frac1\la\left(t-\frac1t\right)V'\!\left(\frac1t\right). 
\]
Therefore 
\begin{equation}\label{V'(t)<->V'(1/t)}
V'(t)=t^{\la-n}\,V'\!\left(\frac1t\right).
\end{equation}
Differentiation of \eqref{f_reflection_potential_t} by $t$ yields 
\begin{eqnarray}
V'(t)&=&\displaystyle (\la-n)t^{\la-n-1}\left[V\left(\frac1t\right)+\frac1\la\left(t-\frac1t\right)V'\!\left(\frac1t\right)\right] \nonumber \\
&&+t^{\la-n}\left[-\frac1{t^2}\,V'\!\left(\frac1t\right)
+\frac1\la\left(1+\frac1{t^2}\right)V'\!\left(\frac1t\right)
+\frac1\la\left(t-\frac1t\right)\left(-\frac1{t^2}\right)V''\!\left(\frac1t\right)
\right] \nonumber \\
&=&\displaystyle \frac{t^{\la-n-2}}{\la}\left\{ -\left(t-\frac1t\right)V''\!\left(\frac1t\right)
+\left[(\la-n)(t^2-1)-\la+(t^2+1)\right]V'\!\left(\frac1t\right) \right. \nonumber \\[1mm]
&& \displaystyle \left. \phantom{\frac{t^{\la-n-2}}{\la}+}+\la(\la-n)\,tV\left(\frac1t\right)
\right\} \nonumber 
\end{eqnarray}
Subsituting \eqref{V'(t)<->V'(1/t)} to the left hand side above and putting $s=1/t$, we obtain 
\[
\la V'(s)=s^2\left\{-\left(\frac1s-s\right)V''(s)+\left[(\la-n)\left(\frac1{s^2}-1\right)-\la+\frac1{s^2}+1\right]V'(s)+\la(\la-n)\frac1sV(s)\right\},
\]
which is equivalent to \eqref{hypgeom_diff__eq_VV'V''}. 

\end{proof}

\begin{theorem}\label{potential_ball}
Let ${V^{(\la)}_{B^n}}(t)$ be the regularized Riesz potential of the unit $n$-ball at a point $x_t=te_1$. When $|t|\ne1$ or $\la>0$ it is given as follows. 
\begin{enumerate}
\item {\rm (cf. \cite{BIK,KM,RS,T})} When $|t|\ne1$ and $\la\ne0$ 
\begin{numcases}
{V^{(\la)}_{B^n}(t)=}
  \frac{\AS}\la\, {}_2F_1\left(-\frac\la 2, -\frac{\la-\n}2;\,\frac n2;\,t^2\right)
   & $(|t|<1)$ \label{V_hypergeom_up}\\ 
 \frac{\AS}n\, t^{\la-\n}\,{}_2F_1\left(-\frac\la 2+1, -\frac{\la-\n}2;\,\frac n2+1;\,\frac1{t^2}\right) 
   & $(|t|>1)$ \label{V_hypergeom_down}
\end{numcases}
where $\AS$ is the volume of the unit $(n-1)$-sphere. 
\item When $|t|\ne1$ and $\la=0$ 
\begin{numcases}
{V^{(0)}_{B^n}(t)=}
{\AS} \, \log\sqrt{1-t^2} 
 & $(|t|<1)$ \label{V^0_up} \\
\AS\biggl(\frac12\log\frac{t+1}{t-1}-\sum_{j=1}^{(n-1)/2}\frac1{2j-1}\,t^{-2j+1}
\biggr) & $(|t|>1, \, n:\mbox{odd })$ \label{V^0_t>1_odd} \\
\AS\biggl(\frac12\log\left(1-\frac1{t^2}\right)-\sum_{j=1}^{n/2-1}\frac1{2j}\,t^{-2j}\biggr) & $(|t|>1 \, n:\mbox{even })$ \label{V^0_t>1_odd} 
\end{numcases}
\item When $|t|=1$ and $\la>0$ 
\begin{equation} \label{V(1)}
V^{(\la)}_{B^n}(1)
=\frac{2^{\la-1}}\la\,\sigma_{n-2}B\left(\frac{\la+1}2,\frac{n-1}2\right)
=\frac{2^\la}\la\,\frac{\Gamma\left(\frac{\la+1}2\right)}{\Gamma\left(\frac{\la+n}2\right)}\,\pi^{\frac{n-1}2}. \notag
\end{equation}
\end{enumerate}
\end{theorem}
\begin{remark}\rm 
We remark that if $|t|<1$ then $\displaystyle \lim_{\la\to\pm0}V_{B^n}^{(\la)}(t)=\pm\infty$ (\cite{O1} Proposition 2.16). 

When $\la\le0$, the limit of $V^{(\la)}_{B^n}(t)$ as $t$ approaches $\pm1$ does not exist. 
It follows from Lemma 2.13 of \cite{O1} that states if $\la\le0$ then 
$\displaystyle \lim_{t\uparrow 1}V_{B^n}^{(\la)}(t)=-\infty$ and $\displaystyle \lim_{t\downarrow 1}V_{B^n}^{(\la)}(t)=\infty$. 
One can apply the regularization using analytic continuation (cf. \cite{OS2}) to $V^{(\la)}_{B^n}(1)$, although we omit an explicit formula. 
\end{remark}

\begin{remark}\label{remark_literature}
\rm 
In some cases when $\la>0$ or $x\in \Om^c$, i.e. when the potential can be defined without regularization as in \eqref{def_potential_la=0_interior} and \eqref{def_potential_la<0_interior}, the formulae \eqref{V_hypergeom_up} and \eqref{V_hypergeom_down} have already appeared in some literatures (the list may not be complete).

\begin{itemize}
\item The case $n=2$ and $0<\la<2$ is given in \cite{KM} Lemma 3.8. 

\item The case $|t|<1$ and $0<\la<1$ is given in \cite{RS} Theorem 4.1. 

\item The case $|t|>1$ is given in \cite{T}, where Tkachev used the inversion in a unit sphere and derived the essentially same differential equation. 

\item The case $0<\la<2$ follows from \cite{BIK} Lemma 2.4, where the Riesz potential of $(1-|y|^2)_+^{\frac\gamma2}$ for $\gamma>0$ was studied, by putting $\gamma=0$. 
\end{itemize}
\end{remark}

\begin{proof}
Put $V(t)={V^{(\la)}_{B^n}}(te_1)$. 

(1) Assume $t\ne\pm1$ and $\la\ne0$. 

(i) Suppose $|t|<1$. 
Put 
\[
F(u)=\frac\AS\la\,{}_2F_1\left(-\frac\la 2, -\frac{\la-n}2;\,\frac n2;\,u\right)
\quad \mbox{ and }\quad
H(t)=F\big(t^2\big). 
\]
Substitution of $a=-\la/2, b=-(\la-n)/2, c=n/2$ and $u=t^2$ to \eqref{gaussian_hypergeom_eq} yields 
\[
\begin{array}{rcl}
0&=&\displaystyle u(1-u)F''(u)+\left\{\frac n2-\left[-\frac\la2-\frac{\la-n}2+1\right]u\right\}F'(u)-\left(-\frac\la2\right)\left(-\frac{\la-n}2\right)F(u)\\[3mm]
&=&\displaystyle t^2(1-t^2)\frac1{4{}\,t^2}\left(H''(t)-\frac1tH'(t)\right)
+\left\{\frac n2+\frac{2\la-n-2}2\,t^2\right\}\frac{1}{2{}\,t}H'(t) 
-\frac{\la(\la-n)}4\frac{}{}H(t) \\[3mm]
&=&\displaystyle \frac1{4{}\,t}\Big[t(1-t^2)H''(t)+\left(n-1+(2\la-n-1)\,t^2\right)H'(t)-\la(\la-n)\,tH(t)\Big]
\end{array}
\]
for $t\ne0$. 
When $t=0$ we have $H'(0)=0$. 
It follows that $H$ satisfies the same differential equation as \eqref{hypgeom_diff__eq_VV'V''} for $V$. 
Since
\[
V(0)=H(0)=\frac{\AS}\la,\>\>V'(0)=H'(0)=0, 
\]
it follows that $V(t)=H(t)$ for any $t$ with $|t|<1$.

\smallskip
(ii) Assume $|t|>1$. 
Since $|\,1/t\,|<1$, by substituting (\ref{V_hypergeom_up}) and \eqref{hypergeom_derivative} to \eqref{f_reflection_potential_t} we obtain 
\[
\begin{array}{rcl}
\displaystyle \frac{V(t)}{\AS}
&=&\displaystyle \frac{t^{\la-n}}\la {}_2F_1\left(-\frac\la 2, -\frac{\la-n}2;\frac n2;\frac1{t^2}\right) \\[3mm]
&& \displaystyle +\frac{\la-n}{\la n}t^{\la-n-2}\left(t^2-1\right)\, {}_2F_1\left(-\frac\la 2+1, -\frac{\la-n}2+1;\frac n2+1;\frac1{t^2}\right)\\[3mm]
&=&\displaystyle \frac{t^{\la-n-2}}{\la n}\left[ nt^2\sum_{i=0}^{\infty}
\frac{{\left(-\frac\la2\right)}_i{\left(-\frac{\la-n}2\right)}_i}{i!\,{\left(\frac n2\right)}_i}\,t^{-2i}
+(\la-n)\left(t^2-1\right)\sum_{j=0}^{\infty}
\frac{{\left(-\frac\la2+1\right)}_j{\left(-\frac{\la-n}2+1\right)}_j}{j!\,{\left(\frac n2+1\right)}_j}\,t^{-2j}
\right] \\[4mm]
%
&=&\displaystyle \frac{t^{\la-n}}{\la n}\left[\la+\sum_{k=1}^{\infty} c_k(\la,n)\,t^{-2k}\right],
\end{array}
\]
where 
\[
\begin{array}{rcl}
c_k(\la,n)&=&\displaystyle n\,\frac{{\left(-\frac\la2\right)}_k{\left(-\frac{\la-n}2\right)}_k}{k!\,{\left(\frac n2\right)}_k}
+(\la-n)\frac{{\left(-\frac\la2+1\right)}_k{\left(-\frac{\la-n}2+1\right)}_k}{k!\,{\left(\frac n2+1\right)}_k} \\[4mm]
&&
\displaystyle -(\la-n)\frac{{\left(-\frac\la2+1\right)}_{k-1}{\left(-\frac{\la-n}2+1\right)}_{k-1}}{(k-1)!\,{\left(\frac n2+1\right)}_{k-1}}\\[5mm]
&=& \displaystyle \frac{{\left(-\frac\la2+1\right)}_{k-1}{\left(-\frac{\la-n}2+1\right)}_{k-1}}{k!\,{\left(\frac n2\right)}_{k+1}} \cdot \frac{\la(\la-2k)(\la-n)n}8  \\ [4mm]
%
%
&=& \displaystyle \la\, \frac{{\left(-\frac\la2+1\right)}_k{\left(-\frac{\la-n}2\right)}_k}{k!\,{\left(\frac n2+1\right)}_k}\,. 
\end{array}
\]
It follows that 
\[
\frac{V(t)}{\AS}=\frac{t^{\la-n}}n\, {}_2F_1\left(-\frac\la 2+1, -\frac{\la-n}2;\,\frac n2+1;\,\frac1{t^2}\right),
\]
which completes the proof of \eqref{V_hypergeom_down}. 

\medskip
(2) 
Let $x\not\in\pO$ 
and $y\in S^{n-1}$. 
Since 
\[
\log |x-y|=\lim_{\la\to0}\frac{{|x-y|}^\la-1}\la,
\]
\eqref{V0_boundary} implies
\begin{eqnarray}
V^{(0)}_{B^n}(x)&=&\displaystyle \int_{S^{n-1}}\frac{\log|x-y|}{{|x-y|}^n}\,(y-x)\cdot \nu\,dy \notag \\
&=&\displaystyle \lim_{\la\to0}\left(
\frac1\la\int_{S^{n-1}}\frac{{|x-y|}^\la}{{|x-y|}^n}\,(y-x)\cdot \nu\,dy 
-\frac1\la\int_{S^{n-1}}\frac{(y-x)\cdot \nu}{{|x-y|}^n}\,dy 
\right). \label{V0_ball}
\end{eqnarray}

(i) Suppose $|t|<1$. 
As 
\begin{equation}\label{div_zero}
\mbox{div}_y\left({|y-x|}^{-n}(y-x)\right)=0, 
\end{equation}
for sufficiently small $\e>0$ we have 
\begin{eqnarray}
0&=&\displaystyle \int_{B^n\setminus B^n_\e(x)}\mbox{div}_y\left(\frac{y-x}{{|x-y|}^n}\right)dy \nonumber \\
&=&\displaystyle \int_{S^{n-1}}\frac{(y-x)\cdot \nu}{{|x-y|}^n}\,dy
-\int_{S^{n-1}_\e(x)}\frac{(y-x)\cdot \nu}{{|x-y|}^n}\,dy \nonumber \\
&=&\displaystyle \int_{S^{n-1}}\frac{(y-x)\cdot \nu}{{|x-y|}^n}\,dy \label{f_integral_sigma_or_0}
-\AS.\end{eqnarray}

By \eqref{V_alpha_boundary}, \eqref{V0_ball}, \eqref{f_integral_sigma_or_0}, and (\ref{V_hypergeom_up}) 
\[
\begin{array}{rcl}
V^{(0)}_{B^n}(t)&=&\displaystyle \lim_{\la\to0}\left(V^{(\la)}_{B^n}(x_t)-\frac\AS\la\right) \\[4mm]
&=&\displaystyle \lim_{\la\to0}\frac\AS\la\left({}_2F_1\left(-\frac\la2,\,-\frac{\la-n}2;\,\frac n2;\,t^2\right)-1 \right) \\[4mm]
&=&\displaystyle -\frac\AS2\lim_{\la\to0}\,\sum_{j=1}^{\infty}\frac{\left(-\frac\la2+1\right)_{j-1}\left(-\frac{\la-n}2\right)_j}{j!\left(\frac n2\right)_j}\,t^{2j} \\[4mm]
&=&\displaystyle -\frac\AS2\sum_{j=1}^{\infty}\frac{(j-1)!}{j!}\,t^{2j} \\[5mm]
%
%
&=&\displaystyle \frac\AS2\log\left(1-t^2\right).
\end{array}
\]

\medskip
(ii) Suppose $|t|>1$. 
Since \eqref{div_zero} implies that the second term of \eqref{V0_ball} vanishes, 
\[
V^{(0)}_{B^n}(t)= \lim_{\la\to0} V^{(\la)}_{B^n}(x_t) = \frac{\AS}{n\,t^n}\, {}_2F_1\left(1, \frac{n}2;\,\frac n2+1;\,\frac1{t^2}\right)
=\AS\sum_{j=0}^\infty \frac1{\n+2j}\,t^{-(\n+2j)},
\]
which implies  
\[
\left(\frac1\AS V^{(0)}_{B^n}(t)\right)'=-\frac{t^{-n-1}}{1-t^{-2}}\,.
\]
Now the conclusion follows from 
\[
\left(\frac12\log\frac{t+1}{t-1}\right)'=-\frac{t^{-2}}{1-t^{-2}} \>\>\mbox{ and }\>\>
\left(\frac12\log\left(1-\frac1{t^2}\right)\right)'=\frac{t^{-3}}{1-t^{-2}}. 
\]

\medskip
(3) Suppose $\la>0$. Then $V_{B^\n}^{(\la)}(x_t)$ is continuous with respect to $t$. 
Since
\[
\frac n2-\left(-\frac\la 2\right)-\left(-\frac{\la-n}2\right)
=\frac n2+1-\left(-\frac\la 2+1\right)-\left(-\frac{\la-n}2\right)
=\la>0,
\]
\eqref{V_hypergeom_up}  
and the Gauss theorem \eqref{Gauss_thm} implies that  
\[
V^{(\la)}_{B^n}(1)=\frac{\AS}\la\, {}_2F_1\left(-\frac\la 2, -\frac{\la-n}2;\,\frac n2;\,1\right)
=\frac\AS\la \frac{\Gamma\left(\frac n2\right)\Gamma(\la)}{\Gamma\left(\frac{\la+n}2\right)\Gamma\left(\frac\la2\right)}
=\frac\AS2 \frac{\Gamma\left(\frac n2\right)\Gamma(\la)}{\Gamma\left(\frac{\la+n}2\right)\Gamma\left(\frac\la2+1\right)}.
\]
Legendre's duplication formula states 
\begin{equation}\label{Legendre_duplication_f}
\sqrt\pi \, \Gamma(2z)=2^{2z-1}\Gamma(z)\Gamma\left(z+\frac12\right),  \notag
\end{equation}
which, substituting $z=(\la+1)/2$, yields 
\[
\sqrt\pi\,\Gamma(\la+1)=2^{\la}\Gamma\left(\frac{\la+1}2\right)\Gamma\left(\frac\la2+1\right),
\]
which, together with $\AS=2\pi^{\n/2}/\Gamma(\n/2)$, implies 
\[
V^{(\la)}_{B^n}(1)
=\frac\AS2 \frac{\Gamma\left(\frac n2\right)\Gamma(\la)}{\Gamma\left(\frac{\la+n}2\right)\Gamma\left(\frac\la2+1\right)}
=\frac{\Gamma(\la)}{\Gamma\left(\frac{\la+n}2\right)\Gamma\left(\frac\la2+1\right)}\,\pi^{\frac n2}
=\frac{2^\la}\la\,\frac{\Gamma\left(\frac{\la+1}2\right)}{\Gamma\left(\frac{\la+n}2\right)}\,\pi^{\frac{n-1}2}.
\]

\end{proof}

\begin{proposition}\label{prop_elementary_ball}
The regularized Riesz potential of the unit $n$-ball at point $x_t=te_1$, $V^{(\la)}_{B^n}(t)$ $(|t|\ne1)$ can be expressed by elementary functions if at least one of the following conditions is satisfied. 
\begin{enumerate}
\item $\la$ is an even integer. 
\item $n$ is odd. 
\end{enumerate}
\end{proposition}

We remark that we have assumed that $\n\ge2$ in this article, although the statement holds even if $\n=1$. 

\begin{proof}
The case when $\la=0$ follows from Theorem \ref{potential_ball} (2). 
In what follows we may assume $|t|<1$ and $\la>0$. The first assumption is justified by Lemma \ref{lemma_reflection_t} and the second by Lemma \ref{prop_V_ball_reflection_alpha}. 

\smallskip
(1) If $\la$ is an even natural number then $-\frac\la2$ is a negative integer, hence ${}_2F_1\left(-\la/2, -(\la-n)/2;\,n/2;\,t^2\right)$ in \eqref{V_hypergeom_up} is a polynomial of $t^2$. 

\medskip
(2) Assume $\la\ne0$ and $n$ is odd. By \eqref{f_V_0} 
\[\begin{array}{rcl}
\la V(t)&=&\displaystyle \int_{S^{n-1}}{\left(t^2+1-2ty_1\right)}^{(\la-n)/2}(1-ty_1)\,dy \\ [4mm]
&=&\displaystyle \int_0^\pi\sigma_{n-2}\sin^{n-2}\theta 
{\left(t^2+1-2t\cos\theta\right)}^{(\la-n)/2}(1-t\cos\theta)\,d\theta \\[4mm]
&=&\displaystyle \sigma_{n-2}\int_0^\pi (1-t\cos\theta){(\sin^2\theta)}^{(n-3)/2}
\sin\theta {\left(t^2+1-2t\cos\theta\right)}^{(\la-n)/2}\,d\theta.
\end{array}\]
Since
\[\begin{array}{rcl}
1-t\cos\theta&=&\displaystyle \frac12\left(t^2+1-2t\cos\theta\right)-\frac12\left(t^2-1\right), \\[4mm]
\sin^2\theta &=& \displaystyle 1-\frac1{4t^2}{\left[
\left(t^2+1-2t\cos\theta\right)-(t^2+1)
\right]}^2
\end{array}\]
and $(\n-3)/2$ is a non-negative integer, 
$\la V(t)/\sigma_{n-2}$ can be expressed as the sum of terms of the form 
\begin{equation}\label{f_beta}
\int_0^\pi q(t^2)\sin\theta {\left(t^2+1-2t\cos\theta\right)}^\beta\,d\theta,
\end{equation}
where $q(t^2)$ is a rational function of $t^2$ and $\beta\in\RR$. Since 
\begin{numcases}
{\int_0^\pi \sin\theta {\left(t^2+1-2t\cos\theta\right)}^\beta\,d\theta=} 
\frac{{(t+1)}^{2(\beta+1)}-{(t-1)}^{2(\beta+1)}}{2t(\beta+1)} & (if $\beta\ne-1$), 
\nonumber \\
\frac1t\log\left|\frac{t+1}{t-1}\right| & (if $\beta=-1$), 
\nonumber 
\end{numcases}
$V(t)$ can be expressed by an elementary function. 
\end{proof}

We remark that 
$\log|t+1|/|t-1|$ appears in $V(t)$ when $\n$ and $\la$ are odd integers 
with $|\la|\le n-2$ since 
\[
\beta=\frac{\la-n}2, \frac{\la-n}2+1, \dots, \frac{\la-n}2+1+n-3.
\]

If $\la-n+2$ is an even natural number then $-(\la-\n)/2$ is a non-positive integer, hence \eqref{V_hypergeom_up} implies that 
$V^{(\la)}_{B^n}(t)$ is a polynominal of $t^2$ for $|t|<1$. 

When $\la=2$, which is the case of Newton potential when $n\ge3$, 
\[
V^{(2)}_{B^n}(t)=\left\{
\begin{array}{ll}
\displaystyle \left(1-\frac{n-2}{n}t^2\right)\frac{\AS}{2} &\hspace{0.5cm}(|t|<1),\\[4mm]
\displaystyle \frac{\AS}{nt^{n-2}} &\hspace{0.5cm}(|t|>1) .
\end{array}
\right.
\]

\begin{proposition}
The log potential satisfies 
\[
V^{\log}_{\Omega}(x)=-\left. \frac\partial{\partial\la}\,{V}^{(\la)}_{\Omega}(x)\right|_{\la=n}
\]
for $x\not\in\pO$. 
\end{proposition}

\begin{proof} The log potential has a boundary integral expression 
\begin{equation}\label{f_last_3}
V^{\log}_{\Omega}(x)=-\frac1n\int_{\partial\Omega}\left(\log|x-y|-\frac1n\right)(y-x)\cdot\nu\,dy. 
\end{equation}
(\cite{O1}, Subsesction 2.7). 
First remark  
\begin{equation}\label{f_last_4}
\int_{\partial\Omega} (y-x)\cdot\nu\,dy =nV^{(n)}_\Omega(x)=n\mbox{Vol}(\Omega).
\end{equation}
Hence \eqref{V_alpha_boundary} implies 
\begin{eqnarray}
\displaystyle \int_{\partial\Omega}\log|x-y| (y-x)\cdot\nu\,dy 
&=&\displaystyle \lim_{\la\to n} \int_{\partial\Omega} \frac{{|x-y|}^{\la-n}-1}{\la-n} \,(y-x)\cdot\nu\,dy \nonumber \\
&=&\displaystyle \lim_{\la\to n} \frac{\la V^{(\la)}_\Omega(x)-nV^{(n)}_\Omega(x)}{\la-n} \nonumber \\
&=&\displaystyle \lim_{\la\to n} \frac{(\la-n)V^{(n)}_\Omega(x)+\la\left(V^{(\la)}_\Omega(x)-V^{(n)}_\Omega(x)\right)}{\la-n} \nonumber \\
&=&\displaystyle \mbox{Vol}(\Omega)+n\lim_{\la\to n} \frac{V^{(\la)}_\Omega(x)-V^{(n)}_\Omega(x)}{\la-n} \nonumber \\[1mm]
&=&\displaystyle \mbox{Vol}(\Omega)+n \left. \frac\partial{\partial\la}\,{V}^{(\la)}_{\Omega}(x)\right|_{\la=n}. \label{f_last_5}
\end{eqnarray}
Substitution of \eqref{f_last_4} and \eqref{f_last_5} to \eqref{f_last_3} implies the conclusion. 
\end{proof}

Department of Mathematics and Informatics,Faculty of Science, 
Chiba University

1-33 Yayoi-cho, Inage, Chiba, 263-8522, JAPAN.  

E-mail: ohara@math.s.chiba-u.ac.jp

\end{document}